\documentclass[11pt]{article}
\usepackage{template}
\newcommand{\MP}{\mathrm{MP}}
\renewcommand{\supp}{\operatorname{supp}}

\title{\Large\bfseries
Integrability of Freely Infinitely Divisible Distributions and L{\'e}vy Measures}
\author{Yu Kitagawa\thanks{Department of Mathematics,
Faculty of Science,
Kyoto University,
Kitashirakawa Oiwake-cho, Sakyo-ku,
Kyoto 606-8502, Japan\\
Email address: kitagawa.yu.57z@st.kyoto-u.ac.jp}}
\date{}

\begin{document}
\maketitle

\begin{abstract}
Under a growth condition on an increasing function $g$, we prove that integrability of a freely infinitely divisible distribution with respect to $g$ is equivalent to that of the large-jump part of its free L{\'e}vy measure. For every increasing freely submultiplicative function $g$, integrability of the distribution implies integrability of the large-jump part of its free L{\'e}vy measure. We also obtain two-sided tail comparisons between freely infinitely divisible distributions and their free L{\'e}vy measures, uniform integrability along free convolution semigroups, and analogous results for fractional free convolution powers.
\end{abstract}

\tableofcontents

\section{Introduction}\label{sec:introduction}

Infinitely divisible distributions and L{\'e}vy processes are among the fundamental objects of classical probability theory. Their free counterparts have been studied since the early development of free probability, beginning with free additive convolution and its analytic transforms and continuing with free L{\'e}vy--Khintchine representations and free L{\'e}vy processes \cite{Voiculescu1986,BV1992,VDN1992,BV1993free,Biane1998,BNT2002,BnT2005Levy,BNT2006}. A central connection between classical and free probability is the Bercovici--Pata bijection \cite{BP1999}, which relates classically and freely infinitely divisible laws through their characteristic triplets.

We recall the two forms of the free L{\'e}vy--Khintchine representation used throughout the paper. If $\mu$ is freely infinitely divisible, its Voiculescu transform admits the representation
\begin{equation}\label{eq:free-generating-pair}
\phi_\mu(z)=\gamma+\int_\RR\frac{1+xz}{z-x}\di{\sigma}(x),\quad z\in\CC^+,
\end{equation}
where $\gamma\in\RR$ and $\sigma$ is a finite positive measure. The pair $(\gamma,\sigma)$ is called the free generating pair, and $\sigma$ is called the free generating measure. Equivalently, the free cumulant transform $C_\mu^{\boxplus}(z)=z\phi_\mu(1/z)$ has the form
\begin{equation}\label{eq:free-triplet}
C_\mu^{\boxplus}(z)=\eta z+az^2+\int_\RR\left(\frac{1}{1-zx}-1-zx\1_{[-1,1]}(x)\right)\di{\nu}(x),\quad z\in\CC^-,
\end{equation}
where $\eta\in\RR$, $a\ge0$, and $\nu$ is a L{\'e}vy measure, that is, $\nu(\{0\})=0$ and
$\int_\RR\min\{1,x^2\}\di{\nu}(x)<\infty$. The parameters $(\eta,a,\nu)$ form the free characteristic triplet of $\mu$, and $\nu$ is its free L{\'e}vy measure.

The free generating measure and the free L{\'e}vy measure are related by
\begin{equation}\label{eq:pair-triplet}
\sigma(\{0\})=a,\quad \di{\sigma}(x)=\frac{x^2}{1+x^2}\di{\nu}(x)\quad\text{on }\RR\setminus\{0\}.
\end{equation}
We write $\ID^{\boxplus}(\RR)$ for the class of freely infinitely divisible probability measures on $\RR$. The Bercovici--Pata bijection sends a classically infinitely divisible law with characteristic triplet $(\eta,a,\nu)$ to the freely infinitely divisible law with the same triplet.

Several structural properties of this bijection are known. For probability measures $\mu_n$ and positive integers $k_n\uparrow\infty$, \cite[Theorem~3.4]{BP1999} shows that weak convergence of $\mu_n^{*k_n}$ to a classically infinitely divisible law is equivalent to weak convergence of $\mu_n^{\boxplus k_n}$ to its image under this bijection. It intertwines classical and free convolution, commutes with affine transformations, and is a homeomorphism with respect to weak convergence \cite[Theorem~3.5 and Corollaries~3.6 and~3.9]{BNT2002}. The bijection also maps classically selfdecomposable laws onto freely selfdecomposable laws \cite[Theorem~4.8]{BNT2002}. Random-matrix models realizing the bijection were constructed in \cite{BG2005} and \cite{CD2005}.

Moment and integrability questions arise throughout free probability. Finite moments are closely related to Taylor expansions of the $R$-transform \cite{BG2006}, and logarithmic integrability is important in the study of Fuglede--Kadison determinants and Brown measures along free convolution semigroups \cite{KPZ2026}. On the classical side, integrability of an infinitely divisible law with respect to a locally bounded submultiplicative function is characterized by integrability of that function against the large-jump part of the L{\'e}vy measure; see \cite[Theorem~25.3]{Sato2013}.

Partial analogues of this classical result were previously known in free probability. For every positive integer $p$, finiteness of the $p$th moment of the free generating measure implies finiteness of the $p$th moment of the freely infinitely divisible law \cite[Proposition~2.3]{BG2006}. The converse was proved there when $p$ is even. Moreover, the compactness of the support of the law is equivalent to that of the generating measure \cite{BV1992,HP2000}.

In free probability, a broader framework was introduced in \cite{CD2005}. For a finite positive measure $\lambda$ on $\RR$ and a nonnegative function $g$ on $[0,\infty)$, write $M_g(\lambda)=\int_\RR g(|x|)\di{\lambda}(x)$. A function $g\colon[0,\infty)\to[0,\infty)$ is called freely submultiplicative if there exists $K_g>0$ such that
\begin{equation*}\tag{FSM}\label{eq:free-submult}
M_g(\mu_1\boxplus\mu_2)\le K_gM_g(\mu_1)M_g(\mu_2)
\end{equation*}
for all probability measures $\mu_1,\mu_2$ on $\RR$. Testing \eqref{eq:free-submult} on Dirac measures shows that
\begin{equation*}\tag{SM}\label{eq:ordinary-submult}
g(r+s)\le K_gg(r)g(s),\quad r,s\ge0,
\end{equation*}
so every freely submultiplicative function is submultiplicative (up to a constant). The functions $1+r^p$ for $p>0$, $e^r$, and $\log(e+r)$ are freely submultiplicative \cite[Proposition~V.2]{CD2005}. In this framework, \cite[Theorem~V.6]{CD2005} asserts that integrability of the free generating measure implies integrability of the law for every locally bounded freely submultiplicative function. However, as explained in \Cref{rem:CD-gap}, the published proof has a gap.

Our first main result establishes the reverse implication for increasing freely submultiplicative functions.

\begin{theorem}\label{thm:intro-law-to-levy}
Let $\mu\in\ID^{\boxplus}(\RR)$ have free generating pair $(\gamma,\sigma)$ and free characteristic triplet $(\eta,a,\nu)$. Let $g\colon[0,\infty)\to[1,\infty)$ be increasing and freely submultiplicative. Then
\[M_g(\mu)<\infty\quad\Longrightarrow\quad M_g(\sigma)<\infty\quad\Longleftrightarrow\quad\int_{|x|>1}g(|x|)\di{\nu}(x)<\infty.\]
\end{theorem}

Here and below we assume $g\ge1$. Replacing a nonnegative function $g$ by $1+g$ does not change any integrability condition considered in this paper; if $g$ is freely submultiplicative, then so is $1+g$, possibly with a different constant. Similarly, changes on a compact interval do not affect any of the integrability conditions considered here. Thus any locally bounded function that is eventually increasing may be modified on a compact interval and treated as globally increasing.

Next we introduce a class of functions for which all three finiteness conditions are equivalent. Let $g\colon[0,\infty)\to[1,\infty)$ be increasing, set $D_g(r)=(\int_0^r u/g(u)\di{u})^{1/2}$, and assume that
\begin{equation*}\tag{\(\Delta\)}\label{eq:condition-Delta}
g(r+D_g(r))\le K_\Delta g(r),\quad r\ge0,
\end{equation*}
for some $K_\Delta<\infty$.

\begin{theorem}\label{thm:intro-Delta-equivalence}
Let $\mu\in\ID^{\boxplus}(\RR)$ have free generating measure $\sigma$ and free L{\'e}vy measure $\nu$. If $g\colon[0,\infty)\to[1,\infty)$ is increasing and satisfies \eqref{eq:condition-Delta}, then
\[M_g(\mu)<\infty\quad\Longleftrightarrow\quad M_g(\sigma)<\infty\quad\Longleftrightarrow\quad\int_{|x|>1}g(|x|)\di{\nu}(x)<\infty.\]
Moreover, the class of probability measures with finite $g$-moment is closed under free convolution.
\end{theorem}

The relation between free submultiplicativity and condition \textup{($\Delta$)} is discussed in \Cref{rem:SM-Delta-comparison}. In particular, the condition \textup{($\Delta$)} does not imply free submultiplicativity, while we do not know whether every increasing freely submultiplicative function satisfies \textup{($\Delta$)}.

Combining \Cref{thm:intro-Delta-equivalence} with the classical result \cite[Theorem~25.3]{Sato2013} gives the following integrability equivalence for the Bercovici--Pata bijection.

\begin{corollary}\label{cor:intro-bp}
Let $g\colon[0,\infty)\to[1,\infty)$ be increasing and satisfy \eqref{eq:condition-Delta} and \eqref{eq:ordinary-submult}. If $\mu$ is classically infinitely divisible and $\Lambda(\mu)$ denotes its image under the Bercovici--Pata bijection, then $M_g(\mu)<\infty$ if and only if $M_g(\Lambda(\mu))<\infty$.
\end{corollary}

Two practical sufficient conditions for \textup{($\Delta$)} are given in \Cref{prop:Delta-examples}.

Tail behavior has also been studied in several settings. The free gamma distributions and their tail asymptotics were analyzed in \cite{HT2014}. Free subexponentiality and regularly varying tails were studied in \cite{HM2013}; for free regular infinitely divisible laws, tail equivalence with the free L{\'e}vy measure under the assumption of regular variation was established in \cite{CCH2020}. The free L{\'e}vy measure of the normal distribution was analyzed in detail in \cite{HU2023}.

The proof of \Cref{thm:intro-law-to-levy} also yields the following quantitative comparison. This theorem complements these asymptotic results: it applies to arbitrary freely infinitely divisible laws, at the cost of comparison only up to fixed dilations and multiplicative constants.

\begin{theorem}\label{thm:intro-tail}
Let $\mu\in\ID^{\boxplus}(\RR)$ have free L{\'e}vy measure $\nu$ and free generating measure $\sigma$. There exist a universal constant $C\ge1$ and constants $B_\mu,R_\mu\ge1$ depending on $\mu$ such that, for every $r\ge R_\mu$,
\begin{align*}
\mu(\{|x|>B_\mu r\})\le\nu(\{|x|>r\})\le C\mu(\{|x|>r/2\}).
\end{align*}
Moreover, \eqref{eq:pair-triplet} gives $\frac12\nu(\{|x|>r\})\le\sigma(\{|x|>r\})\le\nu(\{|x|>r\})$ for $r\ge1$, so the same type of comparison holds between $\mu$ and $\sigma$.
\end{theorem}

The integrability results above also have consequences for free convolution semigroups.

\begin{theorem}\label{thm:intro-levy-process}
Let $(X_t)_{t\ge0}$ be a selfadjoint free L{\'e}vy process with marginal laws $\mu_t=\mu^{\boxplus t}$ for $t\ge0$.
\begin{enumerate}[label=\textup{(\roman*)}]
\item If $g$ is increasing and freely submultiplicative, then $M_g(\mu^{\boxplus t_0})<\infty$ for some $t_0>0$ if and only if $M_g(\mu^{\boxplus t})<\infty$ for every $t>0$.
\item If $g$ is increasing and satisfies \eqref{eq:condition-Delta}, then $M_g(\mu^{\boxplus t_0})<\infty$ for some $t_0>0$ if and only if $M_g(\mu^{\boxplus t})<\infty$ for every $t>0$. In this case, for every $T>0$,
\[\lim_{R\to\infty}\sup_{0\le t\le T}\int_{|x|>R}g(|x|)\di{\mu^{\boxplus t}}(x)=0.\]
Moreover, $t\mapsto\int f\di{\mu_t}$ is continuous for every continuous $f$ satisfying $|f(x)|\le K(1+g(|x|))$ for some $K<\infty$.
\end{enumerate}
In particular, for every $p>0$, if the absolute $p$th moment is finite at some positive time, then it is finite at every positive time and depends continuously on time.
\end{theorem}

Finally, part of the integrability phenomenon persists without infinite divisibility. For every probability measure $\mu$ on $\RR$, the fractional free convolution powers $\mu^{\boxplus t}$ are defined for $t\ge1$. For $s,t\ge1$, they satisfy $\mu^{\boxplus s}\boxplus\mu^{\boxplus t}=\mu^{\boxplus(s+t)}$ and $(\mu^{\boxplus s})^{\boxplus t}=\mu^{\boxplus st}$. See \cite{NS1996,H2015,ST2022} for related properties.

\begin{theorem}\label{thm:intro-partial}
Let $\mu$ be a probability measure on $\RR$, and let $g\colon[0,\infty)\to[1,\infty)$ be an increasing freely submultiplicative function. Then, for every $t\ge1$,
\begin{equation}\label{eq:intro-partial-moment}
M_g(\mu)<\infty\quad\Longleftrightarrow\quad M_g\left(\mu^{\boxplus t}\right)<\infty.
\end{equation}
Moreover, for every $T\ge1$ there exists $C_T\ge1$ such that, for every $1\le t\le T$ and $r>0$,
\begin{equation}\label{eq:intro-partial-tail}
C_T^{-1}\mu(\{|x|>C_Tr\})\le\mu^{\boxplus t}(\{|x|>r\})\le C_T\mu(\{|x|>r/C_T\}).
\end{equation}
If the conditions in \eqref{eq:intro-partial-moment} hold, then for every $T\ge1$,
$\lim_{R\to\infty}\sup_{1\le t\le T}\int_{|x|>R}g(|x|)\di{\mu^{\boxplus t}}(x)=0$.
\end{theorem}

\subsection*{Outline}
\Cref{sec:preliminaries} recalls the tail estimates for spectral distributions and facts about freely infinitely divisible distributions used below. \Cref{sec:levy-data} proves the integrability and tail results. \Cref{sec:partial-semigroups} treats fractional free convolution powers and uniform integrability for free convolution semigroups.

\subsection*{Acknowledgments}
The author thanks Takahiro Hasebe for helpful discussions and comments. Generative AI tools assisted with exploratory discussions concerning \eqref{eq:condition-Delta} and with language editing. The author independently verified all arguments and is solely responsible for the contents of the paper.

\section{Preliminaries}\label{sec:preliminaries}
We recall several general results on generalized singular numbers of unbounded operators, together with some facts about symmetric freely infinitely divisible distributions that will be used in the proofs. Unless otherwise specified, all operators below are understood to be affiliated with a tracial $W^*$-probability space $(\MM,\tau)$.

We write \(\mu_X\) for the spectral distribution of a selfadjoint
affiliated operator \(X\), and \(\check\lambda\) for the
reflection of a measure \(\lambda\) under \(x\mapsto-x\). We denote
by \(\MP_q\) the Marchenko--Pastur law (or free Poisson law) with rate \(q\).

\subsection{Distribution functions and generalized singular numbers}
We use the framework of unbounded random variables affiliated with finite von Neumann algebras developed in \cite{FK1986,BV1993free}. Let $X$ be an operator affiliated with a tracial $W^*$-probability space $(\MM,\tau)$. Define
\[d_X(t)=\tau(\1_{(t,\infty)}(|X|)),\quad \snum{X}(u)=\inf\{t\ge0\colon d_X(t)\le u\},\quad 0\le u\le1.\]
We also have the variational characterization \cite[Definition~2.1 and Proposition~2.2]{FK1986}
\begin{equation}\label{eq:quantile-compression-formula}
\snum{X}(u)=\inf\{\|Xe\|\colon e\in\MM\text{ is a projection and }\tau(1-e)\le u\}.
\end{equation}
The characterization also implies that, for affiliated operators $X$ and $Y$,
\begin{equation}\label{eq:singular-number-sum}
\snum{X+Y}(u+v)\le\snum{X}(u)+\snum{Y}(v),\quad u,v\ge0,\quad u+v\le1;
\end{equation}
see \cite[Lemma~2.5]{FK1986}. Indeed, choose projections $e$ and $f$ such that $\|Xe\|$ and $\|Yf\|$ are arbitrarily close to the respective infima in \eqref{eq:quantile-compression-formula} and use $p=e\wedge f$, for which $\tau(1-p)\le u+v$.

Using the relation $\snum{Z}(u)\le r\Longleftrightarrow d_Z(r)\le u$, we obtain
\begin{equation}\label{eq:tail-triangle}
d_{X+Y}(r+s)\le d_X(r)+d_Y(s),\quad r,s>0.
\end{equation}
Indeed, the case $d_X(r)+d_Y(s)\ge1$ is immediate, while otherwise \eqref{eq:singular-number-sum} applied at $d_X(r)+\varepsilon$ and $d_Y(s)+\varepsilon$, followed by $\varepsilon\downarrow0$, gives the claim.

The Lebesgue measure of $\{u\in(0,1)\colon\snum{X}(u)>t\}$ is equal to $d_X(t)$ for every $t\ge0$. Thus $u\mapsto\snum{X}(u)$ on $(0,1)$ and $|X|$ with respect to $\tau$ have the same distribution. Consequently, for every nonnegative Borel function $h$,
\begin{equation}\label{eq:equimeasurable}
\tau(h(|X|))=\int_0^1h(\snum{X}(u))\di{u}.
\end{equation}

We shall repeatedly use the following monotonicity property. If $X$ and $Y$ are selfadjoint affiliated operators with $X\le Y$, then, by \cite[Corollary~3.3]{BV1993free},
\begin{equation}\label{eq:operator-comparison}
\mu_X((r,\infty))\le\mu_Y((r,\infty)),\quad r\in\RR.
\end{equation}
For probability measures \(\lambda_1\) and \(\lambda_2\) on \(\RR\), we say that \(\lambda_1\) is stochastically dominated by \(\lambda_2\) if
\begin{align*}
\lambda_1((r,\infty))\le\lambda_2((r,\infty)),\quad r\in\RR.
\end{align*}
Equivalently, one may formulate the definition using closed tails, with \((r,\infty)\) replaced by \([r,\infty)\).

Consequently, for every nonnegative increasing Borel function $h$,
\begin{equation}\label{eq:integral-g-inequality}
\int_\RR h(x)\di{\mu_X}(x)\le\int_\RR h(x)\di{\mu_Y}(x).
\end{equation}
More generally, if finite positive measures $\lambda_1$ and $\lambda_2$ satisfy $\lambda_1([r,\infty))\le \lambda_2([r,\infty))$ (or $\lambda_1((r,\infty))\le \lambda_2((r,\infty))$) for every $r\in\RR$, then
\begin{align}\label{eq:integral-general-g-inequality}
\int_\RR h(x)\di{\lambda_1}(x)\le \int_\RR h(x)\di{\lambda_2}(x)
\end{align}
for every nonnegative increasing Borel function $h$.

We also use the following inequality from \cite[Lemma~3.5]{BV1993free}. Let $X$ be a selfadjoint operator affiliated with a tracial $W^*$-probability space $(\MM,\tau)$, and let $p\in\MM $ be a nonzero projection free from $X$. Set $\alpha=\tau(p)$ and $\tau_p=\alpha^{-1}\tau|_{p\MM p}$, and regard $pXp$ as an affiliated operator in $(p\MM p,\tau_p)$. Then, for every $u>0$,
\begin{equation}\label{eq:compression-tail}
\alpha\tau_p(\1_{(u,\infty)}(|pXp|))\le \tau(\1_{(u,\infty)}(|X|)).
\end{equation}
In particular, when $\mu=\mu_X$ and $\tau(p)=1/t$ for $t\ge1$, the operator $tpXp$ has law $\mu^{\boxplus t}$ in $(p\MM p,\tau_p)$ (see \cite{NS1996,ST2022} for further details). Taking $u=r/t$ in \eqref{eq:compression-tail} gives
\begin{equation}\label{eq:fractional-tail-upper}
\mu^{\boxplus t}(\{|x|>r\})\le t\mu(\{|x|>r/t\}),\quad r>0.
\end{equation}

\subsection{Free regular laws}
\label{subsec:free-regular}

We first recall a characterization of free regular probability measures.

\begin{theorem}[{\cite[Theorem~4.2]{AHS2013}}]
\label{thm:free-regular-characterization}
Let $\theta$ be a probability measure on $\RR$. Then $\theta$ is free regular if and only if $\theta$ is freely infinitely divisible and $\operatorname{supp}(\theta^{\boxplus t})\subseteq[0,\infty)$ for every $t>0$. Equivalently, its free cumulant transform admits the representation
\begin{align*}
C_\theta^{\boxplus}(z)=\eta_\theta^{\prime}z+\int_{(0,\infty)}\frac{zx}{1-zx}\di{\nu_\theta}(x),\quad z\in\CC^-,
\end{align*}
where $\eta_\theta^{\prime}\geq0$, $\nu_\theta$ is supported on $(0,\infty)$, and $\int_{(0,\infty)}(1\wedge x)\di{\nu_\theta}(x)<\infty$.
\end{theorem}

We next recall the correspondence between symmetric freely infinitely divisible laws and free regular laws.

\begin{theorem}[{\cite[Theorem~6.1 and Proposition~6.2]{AHS2013}}]
\label{thm:symmetric-free-regular-correspondence}
Let $\rho$ be a symmetric probability measure on $\RR$. Then $\rho$ is freely infinitely divisible if and only if there exists a unique free regular probability measure $\theta$ such that
\begin{align}
\rho^2=\MP_1\boxtimes\theta.
\label{eq:ahs-relations}
\end{align}
Moreover, if the free characteristic triplet of $\rho$ is $(0,a_\rho,\nu_\rho)$, then $\eta_\theta^{\prime}=a_\rho$ and $\nu_\theta=(\nu_\rho)^2$. Here $\rho^2$ and $(\nu_\rho)^2$ denote the push-forwards of $\rho$ and $\nu_\rho$, respectively, under the map $x\mapsto x^2$.
\end{theorem}

We conclude this subsection by recalling compound free Poisson distributions and free Poisson random measures.

Let $\lambda$ be a finite positive measure on $\RR$. We denote by $\pi_\lambda$ the compound free Poisson distribution with jump measure $\lambda$, characterized by
\begin{align*}
C_{\pi_\lambda}^{\boxplus}(z)=\int_\RR \frac{zx}{1-zx}\di{\lambda}(x),
\qquad z\in\CC^-.
\end{align*}
If $\lambda$ is supported on $[0,\infty)$, then $\pi_\lambda$ is free regular and has free L{\'e}vy measure $\lambda|_{(0,\infty)}$. Moreover, if $\theta$ is a probability measure on $\RR$ and $q\ge0$, then $\pi_{q\theta}=\MP_q\boxtimes\theta$. In particular, $\pi_{q\delta_1}=\MP_q$ (see, e.g., \cite[Remark~4.1]{AHS2013}).

We next collect the properties of free Poisson random measures that will be needed below. More details and proofs can be found in \cite{BnT2005Levy}.

\begin{theorem}[{\cite{BnT2005Levy}}]
\label{thm:free-Poisson-random-measures}
Let $(S,\mathcal E,\lambda)$ be a $\sigma$-finite measure space and
set $\mathcal E_0=\{E\in\mathcal E:\lambda(E)<\infty\}$. There exist a tracial $W^*$-probability space $(\MM,\tau)$ and a mapping $M\colon\mathcal E_0\longrightarrow\MM_+$ with the following properties:
\begin{enumerate}
\item For every $E\in\mathcal E_0$, $\mu_{M(E)}=\MP_{\lambda(E)}$.
\item If $E_1,\ldots,E_m\in\mathcal E_0$ are pairwise disjoint, then $M(E_1),\ldots,M(E_m)$ are freely independent and $M\left(\bigcup_{j=1}^m E_j\right)=\sum_{j=1}^mM(E_j)$.
\item For every real-valued function $f\in L^1(S,\mathcal E,\lambda)$, there is a selfadjoint operator affiliated with $\MM$, denoted by $\int_S f\di{M}$, and the integral is linear in $f$. In particular, if $S=\RR$, $\mathcal E=\mathcal B(\RR)$, $\lambda$ is a finite measure, and $\int_\RR |x|\di{\lambda}(x)<\infty$, then $\mu_{\int_\RR x\,\di{M}(x)}=\pi_\lambda$.
\end{enumerate}
\end{theorem}

\section{Free L{\'e}vy measures and their integrability}\label{sec:levy-data}

This section proves \Cref{thm:intro-law-to-levy,thm:intro-Delta-equivalence,cor:intro-bp,thm:intro-tail}. The main tool is symmetrization together with the properties recalled in \Cref{subsec:free-regular}.

\subsection{Integrability for freely submultiplicative functions}

We begin with a tail estimate for positive compound free Poisson distributions.

\begin{lemma}\label{lem:compound-tail}
Let $\lambda$ be a finite positive measure on $[0,\infty)$, and let $\pi_\lambda$ be the compound free Poisson distribution with jump measure $\lambda$. There is a universal constant $c>0$ such that
\begin{equation}\label{eq:compound-tail-lower}
\pi_\lambda([r,\infty))\ge c\lambda([r,\infty))
\end{equation}
whenever $r>0$ and $\lambda([r,\infty))\le1$. Consequently, for every increasing $h\colon[0,\infty)\to[0,\infty)$,
\[\int_{[0,\infty)}h(x)\di{\pi_\lambda}(x)<\infty\quad\Longrightarrow\quad\int_{(0,\infty)}h(x)\di{\lambda}(x)<\infty.\]
\end{lemma}

\begin{proof}
We use the free Poisson random measures constructed in \cite{BnT2005Levy}. Let $M$ be a free Poisson random measure with intensity $\lambda$. For $R>0$, put $\lambda_R=\lambda|_{[0,R]}$ and $X_R=\int_{[0,R]}x\di{M}(x)$. This integral is well defined in the sense of \cite[Definition~4.4]{BnT2005Levy}, and $X_R$ has distribution $\pi_{\lambda_R}$ by \cite[Corollary~4.5]{BnT2005Levy}. Moreover, $\pi_{\lambda_R}\Rightarrow\pi_\lambda$ as $R\to\infty$.

Fix $r>0$ with $\lambda([r,\infty))\le1$ and take $R>r$. Since $x\1_{(0,R]}(x)\ge r\1_{[r,R]}(x)$, we have $X_R\ge rM([r,R])$. The distribution of $M([r,R])$ is $\MP_{\lambda([r,R])}$. Note that there is a universal $c>0$ such that $\MP_q([1,\infty))\ge cq$ for $0\le q\le1$. Indeed, for $q>0$ the absolutely continuous part of $\MP_q$ has density
\[\frac{\sqrt{(b_q-x)(x-a_q)}}{2\pi x}\1_{[a_q,b_q]}(x),\quad a_q=(1-\sqrt q)^2,\quad b_q=(1+\sqrt q)^2;\] see \cite[Proposition~12.11]{NS2006}. On $[1+\sqrt q/2,1+\sqrt q]$, this density is bounded below by $\sqrt{3q/2}/(4\pi)$, and the interval has length $\sqrt q/2$. It follows that
\begin{equation}\label{eq:truncated-tail-bound}
\pi_{\lambda_R}([r,\infty))\ge\MP_{\lambda([r,R])}([1,\infty))\ge c\lambda([r,R]).
\end{equation}
The Portmanteau theorem and \eqref{eq:truncated-tail-bound} now give
\[\pi_\lambda([r,\infty))\ge\limsup_{R\to\infty}\pi_{\lambda_R}([r,\infty))\ge c\lambda([r,\infty)),\]
which proves \eqref{eq:compound-tail-lower}.

Choose $r_0>0$ so that $\lambda([r,\infty))\le1$ for $r\ge r_0$.
By the preceding tail comparison and \eqref{eq:integral-general-g-inequality}, we obtain
\begin{align*}
\int_{[r_0,\infty)}h(x)\di{\lambda}(x)\le c^{-1}\int_{[r_0,\infty)}h(x)\di{\pi_\lambda}(x)<\infty.
\end{align*}
Since $\lambda$ is finite and $h$ is increasing, $\int_{[0,r_0)}h(x)\di{\lambda}(x)\le h(r_0)\lambda([0,r_0))<\infty$. Hence $\int_{[0,\infty)}h(x)\di{\lambda}(x)<\infty$.
\end{proof}

\begin{lemma}\label{lem:regular-levy}
Let $\theta$ be a free regular probability measure on $[0,\infty)$, and let $\nu_\theta$ be its free L{\'e}vy measure. If $h\colon[0,\infty)\to[0,\infty)$ is increasing, then
\[\int_{[0,\infty)}h(x)\di{\theta}(x)<\infty\quad\Longrightarrow\quad\int_{[1,\infty)}h(x)\di{\nu_\theta}(x)<\infty.\]
\end{lemma}

\begin{proof}
Since $\theta$ is free regular, its free cumulant transform has the form
\begin{equation}\label{eq:free-regular-form}
C_\theta^{\boxplus}(z)=\eta_\theta^{\prime} z+\int_{(0,\infty)}\frac{zx}{1-zx}\di{\nu_\theta}(x),
\end{equation}
where $\eta_\theta^{\prime}\ge0$. The part $\eta_\theta^{\prime} z+\int_{(0,1)}\frac{zx}{1-zx}\di{\nu_\theta}(x)$ in \eqref{eq:free-regular-form} corresponds to a free regular measure, while $\int_{[1,\infty)}\frac{zx}{1-zx}\di{\nu_\theta}(x)$ is the free cumulant transform of the compound free Poisson law with jump measure $\nu_\theta|_{[1,\infty)}$. Hence, by additivity of $C^{\boxplus}$, we may realize a positive operator $Y$ with distribution $\theta$ as $Y=Y_1+Y_2$, where $Y_1,Y_2\ge0$ are freely independent and $Y_2$ has the compound free Poisson distribution with jump measure $\nu_\theta|_{[1,\infty)}$. Thus $0\le Y_2\le Y$, and \eqref{eq:integral-g-inequality} gives
\[\int_{[0,\infty)}h(x)\di{\mu_{Y_2}}(x)\le\int_{[0,\infty)}h(x)\di{\mu_Y}(x)=\int_{[0,\infty)}h(x)\di{\theta}(x)<\infty.\]
Applying \Cref{lem:compound-tail} to $Y_2$ proves the claim.
\end{proof}

\begin{proof}[Proof of \Cref{thm:intro-law-to-levy}]
Let $\check\mu$ and $\check\sigma$ be the reflections of $\mu$ and $\sigma$, respectively. Put $\rho=\mu\boxplus\check\mu$ and $\sigma_{\mathrm{s}}=\sigma+\check\sigma$. Then $\rho$ is symmetric and freely infinitely divisible, with free generating pair $(0,\sigma_{\mathrm{s}})$. Free submultiplicativity gives $M_g(\rho)\le K_gM_g(\mu)M_g(\check\mu)=K_gM_g(\mu)^2<\infty$. It is therefore enough to prove that $M_g(\rho)<\infty$ implies $M_g(\sigma_{\mathrm{s}})<\infty$.

Let $(0,a_\rho,\nu_\rho)$ be the free characteristic triplet of $\rho$, and let $\theta$ be the free regular law in \eqref{eq:ahs-relations}. Define $h(x)=g(\sqrt{x})$ for $x\ge0$. Since $\int h\di{\rho^2}=M_g(\rho)<\infty$ and $\rho^2=\MP_1\boxtimes\theta$ is the compound free Poisson distribution with jump measure $\theta$, \Cref{lem:compound-tail} gives $\int_{[0,\infty)}h\di{\theta}<\infty$.
By \Cref{lem:regular-levy} and $\nu_\theta=(\nu_\rho)^2$,
\[\int_{\{|x|\ge1\}}g(|x|)\di{\nu_\rho}(x)=\int_{[1,\infty)}h(x)\di{\nu_\theta}(x)<\infty.\]
Then \eqref{eq:pair-triplet} gives $M_g(\sigma_{\mathrm{s}})<\infty$. The equivalence with the integrability of the large-jump part of $\nu$ follows directly from \eqref{eq:pair-triplet}.
\end{proof}

\begin{remark}\label{rem:CD-gap}
In the proof of \cite[Theorem~V.6]{CD2005}, a compound free Poisson law is approximated by $\alpha_p^{\boxplus p}$, where $\alpha_p$ is a classical compound Poisson law with intensity $\lambda/p$. In estimating $M_g(\alpha_p)$, the inequality
\[
g(|t_1+\cdots+t_m|)\le K_g^{m-1}\prod_{j=1}^m g(|t_j|)
\]
is also applied when $m=0$. The corresponding term is, however, $g(0)$ rather than $K_g^{-1}$, so this would require $g(0)\le K_g^{-1}$. Since this is not part of the assumptions, the resulting bound on $M_g(\alpha_p^{\boxplus p})$ is not justified uniformly in $p$. We therefore do not know whether the equivalence holds for general increasing freely submultiplicative functions.
\end{remark}

The following closure properties provide many examples of freely submultiplicative functions.

\begin{proposition}\label{prop:fsm-examples}
The following functions and closure operations yield freely submultiplicative functions.
\begin{enumerate}[label=\textup{(\roman*)}]
\item The functions $r\mapsto1+r^p$, $r\mapsto e^{cr}$, and $r\mapsto\log(e+cr)$ are freely submultiplicative for every $p,c>0$.
\item If $g$ is freely submultiplicative and $c>0$, then $r\mapsto g(cr)$ is freely submultiplicative. Every $h$ satisfying $c_1g\le h\le c_2g$ for some $c_1,c_2>0$ is also freely submultiplicative. 
\item Every finite positive linear combination of freely submultiplicative functions is freely submultiplicative.
\end{enumerate}
\end{proposition}

\begin{proof}
Part \textup{(i)} follows from \cite[Proposition~V.2]{CD2005} and \textup{(ii)}. Part \textup{(ii)} follows from $D_c(\mu_1\boxplus\mu_2)=D_c\mu_1\boxplus D_c\mu_2$, where $D_c\mu$ is the push-forward of $\mu$ under $x\mapsto cx$. The assertion for comparable functions follows directly from \eqref{eq:free-submult}. Finally, if $g=\sum_{j=1}^m a_jg_j$ with $a_j>0$, then
\[M_g(\mu_1\boxplus\mu_2)\le\sum_{j=1}^m a_jK_{g_j}M_{g_j}(\mu_1)M_{g_j}(\mu_2)\le\left(\sum_{j=1}^m\frac{K_{g_j}}{a_j}\right)M_g(\mu_1)M_g(\mu_2).\]
\end{proof}

\subsection{Integrability under condition \texorpdfstring{\textup{($\Delta$)}}{(Delta)}}

We begin with the elementary consequences of \textup{($\Delta$)}. Since $D_g$ is increasing, iterating \eqref{eq:condition-Delta} gives
\begin{equation}\label{eq:Delta-multiple-shift}
g\bigl(r+aD_g(r)\bigr)\le K_\Delta^{\lceil a\rceil}g(r),
\qquad a\ge0,\quad r\ge0.
\end{equation}
For fixed $d\ge0$, choose an integer $m\ge1$ such that $mD_g(1)\ge d$. Then \eqref{eq:Delta-multiple-shift} gives $g(r+d)\le K_\Delta^m g(r)$ for $r\ge1$, while monotonicity and $g\ge1$ give $g(r+d)\le g(1+d)g(r)$ for $0\le r\le1$. Hence, for some $K_d<\infty$,
\begin{equation}\label{eq:fixed-shift}
g(r+d)\le K_dg(r),\qquad r\ge0.
\end{equation}

We first prove two auxiliary lemmas.

\begin{lemma}\label{lem:tail-integral}
Let $g\colon[0,\infty)\to[1,\infty)$ be increasing. Then, for every finite positive measure $\xi$ on $(0,\infty)$ and every $r\ge0$,
\[\int_0^r\sqrt{\xi((u,\infty))}\di{u}\le\sqrt{2\int_{(0,\infty)}g(t)\di{\xi}(t)}\,D_g(r).\]
\end{lemma}

\begin{proof}
It suffices to consider $\int g\di{\xi}<\infty$. First suppose that $g$ is continuous, and let $m_g$ be the Stieltjes measure on $[0,\infty)$ given by $\di{m_g}=g(0)\delta_0+\di{g}$, so that $m_g([0,u])=g(u)$. Put $h(u)=\sqrt{\xi((u,\infty))}$ and $K_r(t)=\int_t^r g(u)^{-1}\di{u}$. Since $h$ is decreasing, $h(u)\le g(u)^{-1}\int_{[0,u]}h\di{m_g}$. The Fubini--Tonelli theorem and the Cauchy--Schwarz inequality therefore give
\[\left(\int_0^rh(u)\di{u}\right)^2\le\left(\int_{[0,r]}h^2\di{m_g}\right)\left(\int_{[0,r]}K_r^2\di{m_g}\right).\]
The first factor is at most $\int_{(0,\infty)} g\di{\xi}$. For the second factor, $K_r$ is absolutely continuous, with $K_r(r)=0$ and $K_r'(t)=-1/g(t)$. Since $m_g([0,t])=g(t)$, integration by parts and the Fubini--Tonelli theorem give
\begin{align*}
\int_{[0,r]}K_r(t)^2\di{m_g}(t)&=K_r(r)^2g(r)-\int_0^r g(t)(K_r^2)'(t)\di{t}\\
&=2\int_0^rK_r(t)\di{t}\\
&=2\int_0^r\int_t^r\frac{1}{g(u)}\di{u}\di{t}\\
&=2\int_0^r\frac{u}{g(u)}\di{u}
=2D_g(r)^2.
\end{align*}
This proves the claim for continuous $g$.

For general increasing $g$, extend it by $g(0)$ to $(-\infty,0)$ and set $g_n(x)=n\int_{x-1/n}^xg(s)\di{s}$. Then $g_n$ is continuous and increasing, $1\le g_n\le g$, and $g_n(x)\to g(x)$ at every continuity point of $g$, hence for Lebesgue-a.e. $x$. Since $\int g_n\di{\xi}\le\int g\di{\xi}$ and $D_{g_n}(r)\to D_g(r)$, applying the continuous case to $g_n$ and letting $n\to\infty$ proves the claim.
\end{proof}

We next record a quantile comparison used in the upper bound for compound free Poisson laws.

\begin{lemma}\label{lem:quantile-comparison}
Let $Z\ge0$, let $\xi$ be a finite measure on $(0,\infty)$ with mass $q\le1$, and set $W(x)=\xi((x,\infty))$. Suppose that $\tau(\1_{(0,\infty)}(Z))\le q$ and that a Borel function $T\colon[0,\infty)\to[0,\infty)$ satisfies $\snum{Z}(W(x))\le T(x)$ for $x\ge0$. Then, for every increasing $g\colon[0,\infty)\to[0,\infty)$ and every $d\ge0$,
\[\tau(g(d+Z))\le(1-q)g(d)+\int_{(0,\infty)}g(d+T(x))\di{\xi}(x).\]
\end{lemma}

\begin{proof}
The case $q=0$ is immediate. Let $Y\ge0$ have distribution $(1-q)\delta_0+\xi$. Then $d_Y(x)=W(x)$ for $x\ge0$. For $0<u<q$, note that $d_Y(\snum{Y}(u))\le u$. Since $u\mapsto\snum{Z}(u)$ is decreasing, we obtain
\[
\snum{Z}(u)\le\snum{Z}\bigl(d_Y(\snum{Y}(u))\bigr)\le T(\snum{Y}(u)).
\]
Moreover, $\snum{Z}(u)=0$ for $q\le u\le1$. Hence, using \eqref{eq:equimeasurable} first for $Z$ and then for $Y$,
\begin{align*}
\tau(g(d+Z))
&=\int_0^qg(d+\snum{Z}(u))\di{u}+(1-q)g(d)\\
&\le\int_0^qg(d+T(\snum{Y}(u)))\di{u}+(1-q)g(d)\\
&=(1-q)g(d)+\int_{(0,\infty)}g(d+T(x))\di{\xi}(x).
\end{align*}
\end{proof}

\begin{lemma}\label{lem:positive-free-sum}
Let $g$ be increasing and satisfy \eqref{eq:condition-Delta}, and let $A,B\ge0$ be freely independent with $\tau(g(A))+\tau(g(B))<\infty$. Then $\tau(g(A+B))<\infty$.
\end{lemma}

\begin{proof}
Choose $R>0$ such that $\tau(\1_{(R,\infty)}(A))+\tau(\1_{(R,\infty)}(B))\le1$, and set $C=(A-R)_+$ and $D=(B-R)_+$. Since $A+B\le C+D+2R$, \eqref{eq:operator-comparison} and \eqref{eq:fixed-shift} show that it suffices to prove $\tau(g(C+D))<\infty$.

For $u\ge0$, set $P_u=\1_{(u,\infty)}(C)$, $Q_u=\1_{(u,\infty)}(D)$, $W_C(u)=\tau(P_u)$, and $W_D(u)=\tau(Q_u)$. Let $\xi=\mu_C|_{(0,\infty)}+\mu_D|_{(0,\infty)}$, and put $W(x)=\xi((x,\infty))=W_C(x)+W_D(x)$, $q=\xi((0,\infty))$, and $L=\int g\di{\xi}$. Then $q\le1$ and $L<\infty$.

For a positive affiliated operator $H$ and $x\ge0$, we write $H\wedge x\coloneqq f_x(H)$, where $f_x(t)=\min\{t,x\}$. The projections $P_u$ and $Q_u$ are free, and \cite[Theorem~3.3]{Aubrun2021} implies $\|P_u+Q_u\|\le1+\sqrt{W_C(u)}+\sqrt{W_D(u)}$. Since $C\wedge x+D\wedge x=\int_0^x(P_u+Q_u)\di{u}$ and $\sqrt{W_C}+\sqrt{W_D}\le\sqrt{2W}$, \Cref{lem:tail-integral} gives
\begin{align*}
\|C\wedge x+D\wedge x\|\le x+\sqrt2\int_0^x\sqrt{W(u)}\di{u}\le x+2\sqrt L\,D_g(x).
\end{align*}

Let $e_x=1-(P_x\vee Q_x)$. Since $(C-x)_++(D-x)_+$ vanishes on $e_x$ and $\tau(1-e_x)\le W(x)$, \eqref{eq:quantile-compression-formula} gives $\snum{C+D}(W(x))\le \|(C+D)e_x\|\le \|C\wedge x+D\wedge x\|\le x+2\sqrt LD_g(x)$. Moreover, $\tau(\1_{(0,\infty)}(C+D))\le q$. Therefore, \Cref{lem:quantile-comparison} and \eqref{eq:Delta-multiple-shift} yield
\begin{align*}
\tau(g(C+D))
&\le(1-q)g(0)+\int_{(0,\infty)}g\bigl(x+2\sqrt L\,D_g(x)\bigr)\di{\xi}(x)\le g(0)+K_\Delta^{\lceil2\sqrt L\rceil}L<\infty.
\end{align*}
\end{proof}

\begin{corollary}\label{cor:free-sum-stability}
Let $g$ be increasing and satisfy \eqref{eq:condition-Delta}. If $X,Y$ are free selfadjoint operators and $\tau(g(|X|))+\tau(g(|Y|))<\infty$, then $\tau(g(|X+Y|))<\infty$.
\end{corollary}

\begin{proof}
The operators $|X|$ and $|Y|$ are free, so \Cref{lem:positive-free-sum} gives $\tau(g(|X|+|Y|))<\infty$. Since
$-(|X|+|Y|)\le X+Y\le|X|+|Y|$, the positive and negative tails of $X+Y$ are each dominated by the positive tail of $|X|+|Y|$ by \eqref{eq:operator-comparison}. Hence
$d_{X+Y}(r)\le2d_{|X|+|Y|}(r)$ for $r>0$, and thus $\tau(g(|X+Y|))<\infty$.
\end{proof}

\begin{lemma}\label{lem:positive-cp-moment}
Let $g$ be increasing and satisfy \eqref{eq:condition-Delta}, let $\lambda$ be a finite positive measure on $(0,\infty)$, and let $\pi_\lambda$ be its compound free Poisson law. If $\int g\di{\lambda}<\infty$, then $M_g(\pi_\lambda)<\infty$.
\end{lemma}

\begin{proof}
Put $q=\lambda((0,\infty))$ and $L=\int g\di{\lambda}$. First suppose that $q\le1$. Let $M$ be a free Poisson random measure with intensity $\lambda$. For $R>0$, set $\lambda_R=\lambda|_{(0,R]}$, $q_R=\lambda((0,R])$, and
$X_R=\int_{(0,R]}t\di{M}(t)$. Then $X_R$ has law $\pi_{\lambda_R}$. Put $W_R(u)=\lambda((u,R])$ and $M_{R,u}=M((u,R])$.

For $x\ge0$, the identities $t=(t\wedge x)+(t-x)_+$ and $t\wedge x=\int_0^x\1_{(u,\infty)}(t)\di{u}$ give
\begin{align*}
X_R=\int_0^xM_{R,u}\di{u}+\int_{(0,R]}(t-x)_+\di{M}(t)\eqqcolon Y_{R,x}+Z_{R,x}.
\end{align*}
Note that $\|M_{R,u}\|\le(1+\sqrt{W_R(u)})^2$; see \cite[Proposition~12.11]{NS2006}. Since $W_R\le1$ and $W_R(u)\le\lambda((u,\infty))$, \Cref{lem:tail-integral} yields
\begin{align*}
\|Y_{R,x}\|\le\int_0^x(1+\sqrt{W_R(u)})^2\di{u}\le x+3\int_0^x\sqrt{W_R(u)}\di{u}\le x+3\sqrt{2L}\,D_g(x).
\end{align*}

Moreover, $0\le Z_{R,x}\le RM_{R,x}$. Let $e_{R,x}=1-\1_{(0,\infty)}(M_{R,x})$. Then $Z_{R,x}e_{R,x}=0$. Since $M_{R,x}$ has free Poisson law of rate $W_R(x)\le1$, one has $\tau(1-e_{R,x})=W_R(x)$. Hence \eqref{eq:quantile-compression-formula} gives $\snum{X_R}(W_R(x))\le\|X_Re_{R,x}\|=\|Y_{R,x}e_{R,x}\|\le\|Y_{R,x}\|\le x+3\sqrt{2L}\,D_g(x)$. Similarly, $X_R\le RM((0,R])$, and hence $\tau(\1_{(0,\infty)}(X_R))\le q_R$.

Applying \Cref{lem:quantile-comparison} with $\xi=\lambda_R$ and then \eqref{eq:Delta-multiple-shift}, we obtain
\begin{align*}
M_g(\pi_{\lambda_R})\le(1-q_R)g(0)+\int_{(0,R]}g\bigl(x+3\sqrt{2L}\,D_g(x)\bigr)\di{\lambda}(x)\le g(0)+K_\Delta^{\lceil3\sqrt{2L}\rceil}L.
\end{align*}

Set $\widetilde g(r)\coloneqq\int_0^1g(r+s)\di{s}$. Then $\widetilde g$ is continuous, and \eqref{eq:fixed-shift} gives $g\le\widetilde g\le K_1g$. Therefore, the Portmanteau theorem and the preceding uniform estimate give
\begin{align*}
M_g(\pi_\lambda)
&\le M_{\widetilde g}(\pi_\lambda)\le\liminf_{R\to\infty}M_{\widetilde g}(\pi_{\lambda_R})\le K_1\sup_RM_g(\pi_{\lambda_R})<\infty.
\end{align*}

If $q>1$, take $m=\lceil q\rceil$. The preceding case applies to $\lambda/m$, and $\pi_\lambda=\pi_{\lambda/m}^{\boxplus m}$. The conclusion follows from \Cref{cor:free-sum-stability}.
\end{proof}
\begin{proposition}\label{prop:Delta-forward}
Let $g$ be increasing and satisfy \eqref{eq:condition-Delta}, and let $\mu\in\ID^{\boxplus}(\RR)$ have free generating measure $\sigma$ and free L{\'e}vy measure $\nu$. If $M_g(\sigma)<\infty$, then $M_g(\mu)<\infty$.
\end{proposition}

\begin{proof}
Split the free L{\'e}vy measure into its restrictions to $(1,\infty)$ and $(-\infty,-1)$ and the remaining part. The L{\'e}vy measure of the remaining part is supported in $[-1,1]$, so the corresponding law is compactly supported \cite{BV1992,HP2000,BG2006}. The two large-jump parts have finite $g$-moments by \eqref{eq:pair-triplet} and \Cref{lem:positive-cp-moment}, applied to the positive part and to the reflection of the negative part. Repeated applications of \Cref{cor:free-sum-stability} complete the proof.
\end{proof}

\begin{proof}[Proof of \Cref{thm:intro-Delta-equivalence}]
The equivalence between the L{\'e}vy measure and the generating measure follows from \eqref{eq:pair-triplet}. The implication from the generating measure to the law is \Cref{prop:Delta-forward}. The converse follows by the same symmetrization argument as in the proof of \Cref{thm:intro-law-to-levy}, with \Cref{cor:free-sum-stability} in place of free submultiplicativity. The final assertion is \Cref{cor:free-sum-stability}.
\end{proof}

\subsection{Sufficient conditions for \texorpdfstring{\textup{($\Delta$)}}{(Delta)}}

We record two convenient criteria for \textup{($\Delta$)}. We write $D_g(\infty)=\lim_{r\to\infty}D_g(r)$.

\begin{proposition}\label{prop:Delta-examples}
Let $g\colon[0,\infty)\to[1,\infty)$ be increasing. Each of the following conditions implies \eqref{eq:condition-Delta}:
\begin{enumerate}[label=\textup{(\roman*)}]
\item There exist $C\ge1$ and $r_0\ge0$ such that $g(2r)\le Cg(r)$ for $r\ge r_0$.
\item One has $D_g(\infty)<\infty$, and, for every $d>0$, there exists $C_d\ge1$ such that $g(r+d)\le C_dg(r)$ for $r\ge0$.
\end{enumerate}

Consequently, after modifying them on a bounded interval, the following functions satisfy \eqref{eq:condition-Delta}:
\[(1+r)^p(\log(e+r))^q\quad\text{and}\quad e^{cr^\alpha}(1+r)^p(\log(e+r))^q.\]
In the first family, either $p>0$ and $q\in\RR$, or $p=0$ and $q\ge0$. In the second family, $p,q\in\RR$, $c>0$, and $0<\alpha\le1$.
\end{proposition}

\begin{proof}
Since $g\ge1$, one has $D_g(r)\le r/\sqrt2$, and hence $r+D_g(r)\le2r$. This proves the result under \textup{(i)}, after adjusting the constant on a bounded interval. Under \textup{(ii)}, monotonicity gives $g(r+D_g(r))\le g(r+D_g(\infty))\le C_{D_g(\infty)}g(r)$.

The first family is eventually increasing and satisfies \textup{(i)} in the stated parameter range. For the second family, $D_g(\infty)<\infty$, and the inequality $(r+d)^\alpha\leq r^\alpha+d^\alpha$ for $0<\alpha\le1$ gives the estimate in \textup{(ii)}.
\end{proof}

\begin{remark}
\label{rem:SM-Delta-comparison} The second family above satisfies \eqref{eq:ordinary-submult} when $p,q\ge0$, since $(r+s)^\alpha\le r^\alpha+s^\alpha$, $1+r+s\le(1+r)(1+s)$, and $\log(e+r+s)\le \log(e+r)\log(e+s)$. Hence \Cref{cor:intro-bp} applies to these cases. Allowing negative polynomial exponents need not preserve \eqref{eq:ordinary-submult}. For example, for $\beta>0$, after modifying on a bounded interval to make it increasing, $g(r)=e^{cr}(1+r)^{-\beta}$ satisfies \eqref{eq:condition-Delta}, whereas $g(2r)/g(r)^2\asymp r^\beta$, and thus does not satisfy \eqref{eq:ordinary-submult}.

Conversely, suppose that $g$ is increasing and satisfies \textup{(SM)}. Then $g(r+d)\le K_gg(d)g(r)$ for every $d\ge0$, so \Cref{prop:Delta-examples}\textup{(ii)} shows that $D_g(\infty)<\infty$ implies \textup{($\Delta$)}. Moreover, since $r\mapsto\log(K_gg(r))$ is subadditive, the limit $\beta_g\coloneqq\lim_{r\to\infty}r^{-1}\log g(r)$ exists, and $\beta_g>0$ implies $D_g(\infty)<\infty$. Thus the only possible remaining case is $\beta_g=0$ and $D_g(\infty)=\infty$. Many standard examples in this regime, including $1+r^p$ for $0<p\le2$ and $\log(e+r)$, still satisfy \textup{($\Delta$)} by \Cref{prop:Delta-examples}. We do not know whether every increasing freely submultiplicative function satisfies \eqref{eq:condition-Delta}.
\end{remark}

\begin{proof}[Proof of \Cref{cor:intro-bp}]

Since $g$ is increasing and satisfies \textup{(SM)}, the function
$x\mapsto g(|x|)$ is measurable, locally bounded, and satisfies
\[
g(|x+y|)\le g(|x|+|y|)\le K_g g(|x|)g(|y|).
\]
Hence, by \cite[Theorem~25.3]{Sato2013}, $M_g(\mu)<\infty$ is equivalent to $\int_{|x|>1}g(|x|)\di{\nu}(x)<\infty$, where $\nu$ is the classical L{\'e}vy measure of $\mu$. The image of the Bercovici--Pata bijection has the same L{\'e}vy measure, so \Cref{thm:intro-Delta-equivalence} completes the proof.
\end{proof}

\subsection{Tail comparison between the law and the free L{\'e}vy measure}

We now use the same reduction to compare the tails of a freely infinitely divisible law and its free L{\'e}vy measure.

\begin{proposition}\label{prop:upper-levy-tail}
There is a universal constant $C>0$ such that, for every $\mu\in\ID^{\boxplus}(\RR)$ with free L{\'e}vy measure $\nu$,
\[\nu(\{|x|>2r\})\le C\mu(\{|x|>r\})\]
whenever $r>0$ and $2\nu(\{|x|>2r\})\le1$.
\end{proposition}

\begin{proof}
First let $\theta$ be free regular, and fix $s>0$ such that $\nu_\theta((s,\infty))\le1$. The compound free Poisson part with jump measure $\nu_\theta|_{(s,\infty)}$ is stochastically dominated by $\theta$. Thus there is a universal constant $c_1>0$ such that, for every $t>s$, \eqref{eq:compound-tail-lower} and stochastic domination give $\nu_\theta([t,\infty))\le c_1\theta([t,\infty))$. Letting $t\downarrow s$, we obtain
\begin{equation}\label{eq:regular-upper-tail}
\nu_\theta((s,\infty))\le c_1\theta((s,\infty)).
\end{equation}
Now let $\rho\in\ID^{\boxplus}(\RR)$ be symmetric and suppose $\nu_\rho(\{|x|>s\})\le1$. For the free regular law $\theta$ in \eqref{eq:ahs-relations}, one has $\nu_\theta((s^2,\infty))=\nu_\rho(\{|x|>s\})$ and $\rho^2((s^2,\infty))=\rho(\{|x|>s\})$. Applying \eqref{eq:regular-upper-tail} at $s^2$ and then \eqref{eq:compound-tail-lower} to the compound free Poisson law $\rho^2=\MP_1\boxtimes\theta$ gives
\begin{equation}\label{eq:symmetric-upper-tail}
\nu_\rho(\{|x|>s\})\le c_2\rho(\{|x|>s\}),
\end{equation}
for some universal constant $c_2>0$.
Finally, put $\rho=\mu\boxplus\check\mu$ and $s=2r$. Its free L{\'e}vy measure is $\nu+\check\nu$, so the hypothesis allows us to use \eqref{eq:symmetric-upper-tail}. Moreover, \eqref{eq:tail-triangle} gives $\rho(\{|x|>2r\})\le2\mu(\{|x|>r\})$. Since $\nu_\rho(\{|x|>2r\})=2\nu(\{|x|>2r\})$, the claim follows.
\end{proof}

\begin{proposition}\label{prop:reverse-levy-tail}
Let $\mu\in\ID^{\boxplus}(\RR)$ have free characteristic triplet $(\eta,a,\nu)$. There exists $B_\mu\ge1$ such that, whenever $r\ge1$ and $\nu(\{|x|>r\})\le1$,
\[\mu(\{|x|>B_\mu r\})\le\nu(\{|x|>r\}).\]
\end{proposition}

\begin{proof}
Split the free L{\'e}vy--Khintchine representation at the scale $r\ge1$ and realize $\mu$ as the law of the sum of freely independent selfadjoint operators $Y_r$ and $J_r$, where $J_r$ has the compound free Poisson law with jump measure $\nu|_{\{|x|>r\}}$, while $Y_r$ contains the drift, the semicircular part, and the jumps in $\{|x|\le r\}$.

Let $K_\mu=1+|\eta|+a+\int_{|x|\le1}x^2\di{\nu}(x)+\nu(\{|x|>1\})$. Writing $\kappa_n$ for the $n$th free cumulant, the free cumulants of $Y_r$ satisfy $|\kappa_n(Y_r)|\le(K_\mu r)^n$ for $n\ge1$. Indeed,
\begin{align*}
\kappa_1(Y_r)&=\eta+\int_{1<|x|\le r}x\di{\nu}(x),\\
\kappa_2(Y_r)&=a+\int_{|x|\le r}x^2\di{\nu}(x),\\
\kappa_n(Y_r)&=\int_{|x|\le r}x^n\di{\nu}(x),\quad n\ge3.
\end{align*}
For $n\ge3$, use $|x|^n\le r^{n-2}x^2$ on $|x|\le1$ and $|x|^n\le r^n$ on $1<|x|\le r$. The moment-cumulant formula and the bound $4^n$ for the number of noncrossing partitions of $n$ points then give $\supp\mu_{Y_r}\subseteq[-4K_\mu r,4K_\mu r]$.

The compound free Poisson law of $J_r$ assigns mass $1-\nu(\{|x|>r\})$ to zero, and hence $\mu_{J_r}(\{|x|>r\})\le\nu(\{|x|>r\})$. Together with $\supp\mu_{Y_r}\subseteq[-4K_\mu r,4K_\mu r]$, \eqref{eq:tail-triangle} yields $\mu(\{|x|>(4K_\mu+1)r\})\le\nu(\{|x|>r\})$.
Thus the result holds with $B_\mu=4K_\mu+1$.
\end{proof}

\begin{proof}[Proof of \Cref{thm:intro-tail}]
The statement follows from \Cref{prop:upper-levy-tail,prop:reverse-levy-tail}.
\end{proof}

\section{L{\'e}vy processes and fractional free convolution powers}
\label{sec:partial-semigroups}

We begin with probability measures on $[0,\infty)$ and reduce the general case to that setting.

\begin{lemma}\label{lem:positive-partial-moment}
Let $\lambda$ be a probability measure on $[0,\infty)$ and let $g$ be increasing and freely submultiplicative. For every $t\ge1$,
\[M_g(\lambda)<\infty\quad\Longleftrightarrow\quad M_g(\lambda^{\boxplus t})<\infty.\]
If these conditions hold, then for every $T\ge1$,
\[\lim_{R\to\infty}\sup_{1\le t\le T}\int_{(R,\infty)}g(x)\di{\lambda^{\boxplus t}}(x)=0.\]
\end{lemma}

\begin{proof}
Suppose first that $M_g(\lambda)<\infty$, fix $t\ge1$, and choose an integer $n\ge t+1$. Then $\lambda^{\boxplus n}=\lambda^{\boxplus t}\boxplus\lambda^{\boxplus(n-t)}$. Since all terms are supported on $[0,\infty)$, \eqref{eq:operator-comparison} shows that $\lambda^{\boxplus t}$ is stochastically dominated by $\lambda^{\boxplus n}$. Repeated applications of free submultiplicativity give $M_g(\lambda^{\boxplus n})<\infty$, and the preceding stochastic domination yields $M_g(\lambda^{\boxplus t})<\infty$.

Conversely, if $M_g(\lambda^{\boxplus t})<\infty$, free submultiplicativity gives $M_g(\lambda^{\boxplus2t})<\infty$. Since $\lambda^{\boxplus2t}=\lambda\boxplus\lambda^{\boxplus(2t-1)}$, positivity implies again that $\lambda$ is stochastically dominated by $\lambda^{\boxplus2t}$. Thus, $M_g(\lambda)\le M_g(\lambda^{\boxplus2t})<\infty$.

For uniform integrability, choose an integer $N\ge T+1$. For $1\le t\le T$, positivity in $\lambda^{\boxplus N}=\lambda^{\boxplus t}\boxplus\lambda^{\boxplus(N-t)}$ shows that $\lambda^{\boxplus t}$ is stochastically dominated by $\lambda^{\boxplus N}$. Applying this domination to the increasing truncated function $x\mapsto g(x)\1_{(R,\infty)}(x)$ implies the claim.
\end{proof}

\begin{lemma}\label{lem:absolute-compression}
Let $X$ be selfadjoint with law $\mu$, let $\lambda$ be the law of $|X|$, and let $t\ge1$. Then
\begin{equation}\label{eq:absolute-compression-tail}
\mu^{\boxplus t}(\{|x|>r\})\le2\lambda^{\boxplus t}((r,\infty)),\quad r>0,
\end{equation}
and hence, for every increasing $g\ge1$,
\begin{equation}\label{eq:absolute-compression-moment}
M_g(\mu^{\boxplus t})\le2M_g(\lambda^{\boxplus t}).
\end{equation}
\end{lemma}

\begin{proof}
Let $p$ be free from $X$ with $\tau(p)=1/t$. In the $W^*$-probability space $(p\MM p,t\tau|_{p\MM p})$, the operators $tpXp$ and $tp|X|p$ have laws $\mu^{\boxplus t}$ and $\lambda^{\boxplus t}$, respectively. Since $-tp|X|p\le tpXp\le tp|X|p$, the positive and negative tails of $tpXp$ are each dominated by the positive tail of $tp|X|p$. This proves \eqref{eq:absolute-compression-tail} and then \eqref{eq:absolute-compression-moment}.
\end{proof}

\begin{lemma}\label{lem:symmetric-no-cancellation}
Let $\rho$ be a symmetric probability measure on $\RR$. Then, for every $r>0$,
\begin{equation}\label{eq:symmetric-no-cancellation}
\rho(\{|x|>r\})\le2\rho^{\boxplus2}(\{|x|>r\}).
\end{equation}
Consequently, $M_g(\rho)\le2M_g(\rho^{\boxplus2})$ for every increasing $g\ge0$.
\end{lemma}

\begin{proof}
Let $X,Y$ be free with law $\rho$. Symmetry gives $\mu_{X+Y}=\mu_{X-Y}=\rho^{\boxplus2}$. Since $2X=(X+Y)+(X-Y)$, \eqref{eq:tail-triangle} gives \eqref{eq:symmetric-no-cancellation} and hence the stated inequality.
\end{proof}

\begin{proposition}\label{prop:time-two}
Let $\mu$ be a probability measure on $\RR$. Then
\begin{equation}\label{eq:time-two-tail}
\mu(\{|x|>r\})\le5\mu^{\boxplus2}(\{|x|>r/2\}),\quad r>0.
\end{equation}
If $g$ is increasing and freely submultiplicative, then
\begin{equation}\label{eq:time-two-moment-lower}
M_g(\mu)\le M_g(\mu^{\boxplus2})+2K_gM_g(\mu^{\boxplus2})^2.
\end{equation}
In particular, $M_g(\mu)<\infty$ if and only if $M_g(\mu^{\boxplus2})<\infty$.
\end{proposition}

\begin{proof}
Let $X,Y$ be free with law $\mu$, put $\alpha=\mu^{\boxplus2}$, and set $\rho=\mu\boxplus\check\mu$. From $2X=(X+Y)+(X-Y)$ and \eqref{eq:tail-triangle}, 
\begin{align}
\label{eq:sum-sym-tail-inequality}
\mu(\{|x|>r\})\le\alpha(\{|x|>r\})+\rho(\{|x|>r\}).
\end{align}
By \Cref{lem:symmetric-no-cancellation}, $\rho(\{|x|>r\})\le2\rho^{\boxplus2}(\{|x|>r\})$. Since $\rho^{\boxplus2}=\alpha\boxplus\check\alpha$, another application of \eqref{eq:tail-triangle} gives $\rho^{\boxplus2}(\{|x|>r\})\le2\alpha(\{|x|>r/2\})$. Hence $\mu(\{|x|>r\})\le5\alpha(\{|x|>r/2\})$, which is \eqref{eq:time-two-tail}.
Applying \eqref{eq:integral-general-g-inequality} to \eqref{eq:sum-sym-tail-inequality} gives $M_g(\mu)\le M_g(\alpha)+M_g(\rho)$. By \Cref{lem:symmetric-no-cancellation} and free submultiplicativity,
\[M_g(\rho)\le2M_g(\rho^{\boxplus2})=2M_g(\alpha\boxplus\check\alpha)\le2K_gM_g(\alpha)^2,\]
which proves \eqref{eq:time-two-moment-lower}.
\end{proof}

\begin{proof}[Proof of \Cref{thm:intro-partial}]
Let $\lambda$ be the law of $|X|$ when $X$ has law $\mu$. We first prove the moment equivalence. If $M_g(\mu)<\infty$, then \Cref{lem:positive-partial-moment,lem:absolute-compression} give $M_g(\mu^{\boxplus t})\le2M_g(\lambda^{\boxplus t})<\infty$.

Conversely, suppose $M_g(\mu^{\boxplus t})<\infty$. Choose an integer $k$ with $2^k\ge t$ and put $s=2^k/t\ge1$. Applying the implication just proved to $\mu^{\boxplus t}$ gives $M_g((\mu^{\boxplus t})^{\boxplus s})=M_g(\mu^{\boxplus2^k})<\infty$.
Iterating \Cref{prop:time-two} from time $2^k$ down to time $1$ gives $M_g(\mu)<\infty$.

For the upper tail bound, \eqref{eq:fractional-tail-upper} gives
\[\mu^{\boxplus t}(\{|x|>r\})\le T\mu(\{|x|>r/T\}),\quad1\le t\le T.\]
Choose an integer $k$ with $2^k\ge T$. Iterating \eqref{eq:time-two-tail} gives
\[\mu(\{|x|>R\})\le5^k\mu^{\boxplus2^k}(\{|x|>R/2^k\}).\]
For $1\le t\le T$, set $s=2^k/t\in[1,2^k]$. Applying \eqref{eq:fractional-tail-upper} to the law $\mu^{\boxplus t}$ and the exponent $s$ gives
\[\mu^{\boxplus2^k}(\{|x|>R/2^k\})\le s\mu^{\boxplus t}(\{|x|>tR/4^k\}).\]
Taking $R=4^kr/t$ and using $s\le2^k$ and $t\ge1$, we obtain
\[
\mu^{\boxplus t}(\{|x|>r\})\ge10^{-k}\mu(\{|x|>4^kr/t\})\ge10^{-k}\mu(\{|x|>4^kr\}),
\]
uniformly for $1\le t\le T$.
Together with the upper bound, this yields \eqref{eq:intro-partial-tail} after enlarging a single constant.

For uniform integrability, using \eqref{eq:absolute-compression-tail}, we have
\[\sup_{1\le t\le T}\int_{|x|>R}g(|x|)\di{\mu^{\boxplus t}}(x)\le2\sup_{1\le t\le T}\int_{(R,\infty)}g(x)\di{\lambda^{\boxplus t}}(x),\]
and the right-hand side tends to zero by \Cref{lem:positive-partial-moment}.
\end{proof}

We now turn to free convolution semigroups.

Let $(\mu_t)_{t\ge0}$ be a free convolution semigroup of freely infinitely divisible distributions, and let $(\eta,a,\nu)$ be the free characteristic triplet of $\mu_1$. Thus the triplet of $\mu_t$ is $(t\eta,ta,t\nu)$.

\begin{proof}[Proof of \Cref{thm:intro-levy-process}]
Let $0<s<t$. Since $\mu_t=(\mu_s)^{\boxplus t/s}$, \Cref{thm:intro-partial} gives $M_g(\mu_s)<\infty\quad\Longleftrightarrow\quad M_g(\mu_t)<\infty$, which proves part \textup{(i)}.

For part \textup{(ii)}, assume that $g$ is increasing and satisfies \eqref{eq:condition-Delta}. Since the free L{\'e}vy measure of $\mu_t$ is $t\nu$, \Cref{thm:intro-Delta-equivalence} gives, for every $t>0$,
\begin{align*}
M_g(\mu_t)<\infty
\quad\Longleftrightarrow\quad
\int_{|x|>1}g(|x|)\di{\nu}(x)<\infty.
\end{align*}
Thus finiteness at one positive time is equivalent to finiteness at every positive time. We now assume that these equivalent conditions hold.

Fix $T>0$, and let $\lambda$ be the image of $\nu|_{\{|x|>1\}}$ under $x\mapsto|x|$. As before, realize $\mu_t$ as the law of $A_t+J_t$, where $J_t=P_t^+-P_t^-$ and the variables $A_t,P_t^+,P_t^-$ are freely independent. Here $A_t$ contains the drift, the semicircular part, and the small jumps with free L{\'e}vy measure $t\nu|_{[-1,1]}$, while $P_t^+$ and $P_t^-$ are positive compound free Poisson variables corresponding to $t\nu|_{(1,\infty)}$ and to the image of $t\nu|_{(-\infty,-1)}$ under $x\mapsto-x$, respectively.

As in the proof of \Cref{prop:reverse-levy-tail}, we have $D_T\coloneqq\sup_{0\le t\le T}\|A_t\|<\infty$. Indeed, if $K_T\coloneqq1+T\left(|\eta|+a+\int_{|x|\le1}x^2\di{\nu}(x)\right)$, then the free cumulants of $A_t$ satisfy $|\kappa_n(A_t)|\le K_T^n$ uniformly for $0\le t\le T$ and $n\ge1$. Hence the moment-cumulant formula as in the proof of \Cref{prop:reverse-levy-tail} gives $D_T\le4K_T$.

Put $Q_t=P_t^++P_t^-$. Then $-Q_t\le J_t\le Q_t$ and $\mu_{Q_t}=\pi_{t\lambda}$. Moreover,
$\pi_{T\lambda}=\pi_{t\lambda}\boxplus\pi_{(T-t)\lambda}$, so positivity implies that $\pi_{t\lambda}$ is stochastically dominated by $\pi_{T\lambda}$. Since $g$ is increasing, \eqref{eq:integral-g-inequality} gives
\begin{align*}
\int_{|x|>R}g(|x|)\di{\mu_{J_t}}(x)
\le2\int_{(R,\infty)}g(x)\di{\pi_{T\lambda}}(x).
\end{align*}
Since $\int g\di{\lambda}<\infty$, \Cref{lem:positive-cp-moment} gives $M_g(\pi_{T\lambda})<\infty$.

Since $\|A_t\|\le D_T$, \eqref{eq:tail-triangle} gives $\mu_t(\{|x|>R+D_T\})\le\mu_{J_t}(\{|x|>R\})$ for $R>0$. Hence, by \eqref{eq:integral-general-g-inequality} and \eqref{eq:fixed-shift},
\begin{align*}
\sup_{0\le t\le T}\int_{|x|>R+D_T}g(|x|)\di{\mu_t}(x)
&\le K_{D_T}\sup_{0\le t\le T}\int_{|x|>R}g(|x|)\di{\mu_{J_t}}(x)\\
&\le2K_{D_T}\int_{(R,\infty)}g(x)\di{\pi_{T\lambda}}(x)
\longrightarrow0.
\end{align*}
This proves uniform integrability with respect to $g$ on compact time intervals.

Finally, let $f$ be continuous and satisfy $|f(x)|\le K(1+g(|x|))$. Let $f_R=f\chi_R$, where $0\le\chi_R\le1$ is continuous, compactly supported, and equal to one on $[-R,R]$. Weak continuity gives continuity of $t\mapsto\int f_R\di{\mu_t}$. Since $g\ge1$, uniform integrability gives
\begin{align*}
\sup_{0\le t\le T}
\left|\int f\di{\mu_t}-\int f_R\di{\mu_t}\right|
&\le2K\sup_{0\le t\le T}\int_{|x|>R}g(|x|)\di{\mu_t}(x)
\longrightarrow0.
\end{align*}
Thus $t\mapsto\int f\di{\mu_t}$ is continuous.

For every $p>0$, the function $g(r)=(1+r)^p$ is increasing and satisfies \eqref{eq:condition-Delta}. Taking $f(x)=|x|^p$ proves the final assertion.
\end{proof}
\printbibliography

@article{AHS2013,
  author  = {Arizmendi, Octavio and Hasebe, Takahiro and Sakuma, Noriyoshi},
  title   = {On the Law of Free Subordinators},
  journal = {ALEA Lat. Am. J. Probab. Math. Stat.},
  volume  = {10},
  number  = {1},
  pages   = {271--291},
  year    = {2013}
}

@article{Aubrun2021,
  author  = {Aubrun, Guillaume},
  title   = {Principal Angles between Random Subspaces and Polynomials in Two Free Projections},
  journal = {Confluentes Math.},
  volume  = {13},
  number  = {2},
  pages   = {3--10},
  year    = {2021},
  doi     = {10.5802/cml.74}
}

@article{BG2005,
  author  = {Benaych-Georges, Florent},
  title   = {Classical and Free Infinitely Divisible Distributions and Random Matrices},
  journal = {Ann. Probab.},
  volume  = {33},
  number  = {3},
  pages   = {1134--1170},
  year    = {2005},
  doi     = {10.1214/009117904000000982}
}

@article{BG2006,
  author  = {Benaych-Georges, Florent},
  title   = {Taylor Expansions of {$R$}-Transforms, Application to Supports and Moments},
  journal = {Indiana Univ. Math. J.},
  volume  = {55},
  number  = {2},
  pages   = {465--482},
  year    = {2006},
  doi     = {10.1512/iumj.2006.55.2691}
}

@article{Biane1998,
  author  = {Biane, Philippe},
  title   = {Processes with Free Increments},
  journal = {Math. Z.},
  volume  = {227},
  number  = {1},
  pages   = {143--174},
  year    = {1998},
  doi     = {10.1007/PL00004363}
}

@article{BNT2002,
  author  = {Barndorff-Nielsen, Ole E. and Thorbj{\o}rnsen, Steen},
  title   = {Self-Decomposability and L{\'e}vy Processes in Free Probability},
  journal = {Bernoulli},
  volume  = {8},
  number  = {3},
  pages   = {323--366},
  year    = {2002}
}

@article{BnT2005Levy,
  author  = {Barndorff-Nielsen, Ole E. and Thorbj{\o}rnsen, Steen},
  title   = {The L{\'e}vy--It{\^o} Decomposition in Free Probability},
  journal = {Probab. Theory Related Fields},
  volume  = {131},
  number  = {2},
  pages   = {197--228},
  year    = {2005},
  doi     = {10.1007/s00440-003-0322-y}
}

@incollection{BNT2006,
  author    = {Barndorff-Nielsen, Ole E. and Thorbj{\o}rnsen, Steen},
  title     = {Classical and Free Infinite Divisibility and L{\'e}vy Processes},
  booktitle = {Quantum Independent Increment Processes II},
  editor    = {Franz, Uwe and Sch{\"u}rmann, Michael},
  series    = {Lecture Notes in Mathematics},
  volume    = {1866},
  pages     = {33--159},
  publisher = {Springer},
  address   = {Berlin},
  year      = {2006},
  doi       = {10.1007/11376637_2}
}

@article{BP1999,
  author  = {Bercovici, Hari and Pata, Vittorino},
  title   = {Stable Laws and Domains of Attraction in Free Probability Theory},
  journal = {Ann. of Math. (2)},
  volume  = {149},
  number  = {3},
  pages   = {1023--1060},
  year    = {1999},
  doi     = {10.2307/121080},
  note    = {With an appendix by Philippe Biane}
}

@article{BV1992,
  author  = {Bercovici, Hari and Voiculescu, Dan-Virgil},
  title   = {L{\'e}vy--Hin{\v c}in Type Theorems for Multiplicative and Additive Free Convolution},
  journal = {Pacific J. Math.},
  volume  = {153},
  number  = {2},
  pages   = {217--248},
  year    = {1992},
  doi     = {10.2140/pjm.1992.153.217}
}

@article{BV1993free,
  author  = {Bercovici, Hari and Voiculescu, Dan-Virgil},
  title   = {Free Convolution of Measures with Unbounded Support},
  journal = {Indiana Univ. Math. J.},
  volume  = {42},
  number  = {3},
  pages   = {733--773},
  year    = {1993},
  doi     = {10.1512/iumj.1993.42.42033}
}

@article{CCH2020,
  author  = {Chakrabarty, Arijit and Chakraborty, Sukrit and Hazra, Rajat Subhra},
  title   = {Regular Variation and Free Regular Infinitely Divisible Laws},
  journal = {Statist. Probab. Lett.},
  volume  = {156},
  pages   = {108607},
  year    = {2020},
  doi     = {10.1016/j.spl.2019.108607}
}

@article{CD2005,
  author  = {Cabanal-Duvillard, Thierry},
  title   = {A Matrix Representation of the {Bercovici--Pata} Bijection},
  journal = {Electron. J. Probab.},
  volume  = {10},
  number  = {18},
  pages   = {632--661},
  year    = {2005},
  doi     = {10.1214/EJP.v10-246}
}

@article{FK1986,
  author  = {Fack, Thierry and Kosaki, Hideki},
  title   = {Generalized {$s$}-Numbers of {$\tau$}-Measurable Operators},
  journal = {Pacific J. Math.},
  volume  = {123},
  number  = {2},
  pages   = {269--300},
  year    = {1986},
  doi     = {10.2140/pjm.1986.123.269}
}

@article{H2015,
  author  = {Huang, Hao-Wei},
  title   = {Supports of Measures in a Free Additive Convolution Semigroup},
  journal = {Int. Math. Res. Not. IMRN},
  volume  = {2015},
  number  = {12},
  pages   = {4269--4292},
  year    = {2015},
  doi     = {10.1093/imrn/rnu064}
}

@article{HM2013,
  author  = {Hazra, Rajat Subhra and Maulik, Krishanu},
  title   = {Free Subexponentiality},
  journal = {Ann. Probab.},
  volume  = {41},
  number  = {2},
  pages   = {961--988},
  year    = {2013},
  doi     = {10.1214/11-AOP706}
}

@book{HP2000,
  author    = {Hiai, Fumio and Petz, D{\'e}nes},
  title     = {The Semicircle Law, Free Random Variables and Entropy},
  series    = {Mathematical Surveys and Monographs},
  volume    = {77},
  publisher = {American Mathematical Society},
  address   = {Providence, RI},
  year      = {2000},
  isbn      = {978-0-8218-2081-0}
}

@misc{KPZ2026,
  author        = {Kitagawa, Yu and Popa, Mihai and Zhong, Ping},
  title         = {Freely Infinitely Divisible {$R$}-Diagonal Elements and Brown Measure},
  year          = {2026},
  eprint        = {2605.25434},
  archiveprefix = {arXiv},
  primaryclass  = {math.OA},
}

@article{NS1996,
  author  = {Nica, Alexandru and Speicher, Roland},
  title   = {On the Multiplication of Free {$N$}-Tuples of Noncommutative Random Variables},
  journal = {Amer. J. Math.},
  volume  = {118},
  number  = {4},
  pages   = {799--837},
  year    = {1996},
  doi     = {10.1353/ajm.1996.0034},
  note    = {With an appendix by Dan-Virgil Voiculescu}
}

@book{NS2006,
  author    = {Nica, Alexandru and Speicher, Roland},
  title     = {Lectures on the Combinatorics of Free Probability},
  series    = {London Mathematical Society Lecture Note Series},
  volume    = {335},
  publisher = {Cambridge University Press},
  address   = {Cambridge},
  year      = {2006},
  doi       = {10.1017/CBO9780511735127}
}

@book{Sato2013,
  author = {given=Ken-iti, given-i=K, family=Sato},
  title     = {L{\'e}vy Processes and Infinitely Divisible Distributions},
  series    = {Cambridge Studies in Advanced Mathematics},
  volume    = {68},
  publisher = {Cambridge University Press},
  address   = {Cambridge},
  year      = {2013},
  isbn      = {978-1-107-65649-9},
  edition = {2},
}

@article{ST2022,
  author  = {Shlyakhtenko, Dimitri and Tao, Terence},
  title   = {Fractional Free Convolution Powers},
  journal = {Indiana Univ. Math. J.},
  volume  = {71},
  number  = {6},
  pages   = {2551--2594},
  year    = {2022},
  doi     = {10.1512/iumj.2022.71.9163},
  note    = {With an appendix by David Jekel}
}

@book{VDN1992,
  author    = {Voiculescu, Dan-Virgil and Dykema, Kenneth J. and Nica, Alexandru},
  title     = {Free Random Variables},
  series    = {CRM Monograph Series},
  volume    = {1},
  publisher = {American Mathematical Society},
  address   = {Providence, RI},
  year      = {1992}
}

@article{Voiculescu1986,
  author  = {Voiculescu, Dan-Virgil},
  title   = {Addition of Certain Non-Commuting Random Variables},
  journal = {J. Funct. Anal.},
  volume  = {66},
  number  = {3},
  pages   = {323--346},
  year    = {1986},
  doi     = {10.1016/0022-1236(86)90062-5}
}

@article{HT2014,
  author  = {Haagerup, Uffe and Thorbj{\o}rnsen, Steen},
  title   = {On the Free Gamma Distributions},
  journal = {Indiana Univ. Math. J.},
  volume  = {63},
  number  = {4},
  pages   = {1159--1194},
  year    = {2014},
  doi     = {10.1512/iumj.2014.63.5288}
}

@article{HU2023,
  author  = {Hasebe, Takahiro and Ueda, Yuki},
  title   = {On the Free L{\'e}vy Measure of the Normal Distribution},
  journal = {Electron. J. Probab.},
  volume  = {28},
  number  = {133},
  pages   = {1--19},
  year    = {2023},
  doi     = {10.1214/23-EJP1035}
}

\end{document}